\documentclass[12pt]{amsart}

\usepackage[leqno]{amsmath}
\usepackage{amssymb,amsthm,latexsym,mathtools,dsfont}
\usepackage{mathrsfs}
\usepackage{mdwlist,physics}
\usepackage{graphicx}
\usepackage{enumerate}
\usepackage[shortlabels]{enumitem}
\usepackage{a4wide}

\usepackage{eucal}

\usepackage{physics}
\usepackage{color}
\definecolor{amber}{rgb}{1.0, 0.75, 0.0}

\newcommand{\IR}{\mathbb{R}}

\newcommand{\N}{\mathbb{N}}
\newcommand{\IN}{\mathbb{N}}

\newcommand{\IK}{\mathbb{K}}

\newcommand{\mc}{\mathcal}
\newcommand{\on}{\operatorname}

\newcommand{\FF}{\mathcal{F}}

\newcommand{\vertiii}[1]{{\left\vert\kern-0.25ex\left\vert\kern-0.25ex\left\vert #1 
    \right\vert\kern-0.25ex\right\vert\kern-0.25ex\right\vert}}

\newcounter{smallromans}

{\end{list}}

\newcounter{smallromansdash}

{\end{list}}

\newcounter{bigromans} 
  {\end{list}}

\newcommand{\nn}[1]{{\vert\kern-0.25ex\vert\kern-0.25ex\vert #1 
    \vert\kern-0.25ex\vert\kern-0.25ex\vert}}
\newcommand{\lnn}[1]{{\left\vert\kern-0.25ex\left\vert\kern-0.25ex\left\vert #1 
    \right\vert\kern-0.25ex\right\vert\kern-0.25ex\right\vert}}

\renewcommand{\leq}{\leqslant}

\newcommand{\ou}{%
  \mathrel{%
    \vcenter{\offinterlineskip
      \ialign{##\cr$\forall$\cr\noalign{\kern-1.5pt}$\exists$\cr}%
    }%
  }%
}

\newcounter{maintheorem}

\newtheorem{mainth}[maintheorem]{Theorem}
\newtheorem*{mainthprime*}{Theorem A$^\prime$}

\newtheorem{corollary*}{Corollary}

\theoremstyle{definition}

\theoremstyle{remark}

\newtheorem{question}{Question}

\title[Coordinate functionals for filter bases]{Continuity of coordinate functionals\\ of filter bases in Banach spaces}
\author[N.~de~Rancourt]{No\'{e} de Rancourt}
\address[N.~de~Rancourt]{ Faculty of Mathematics and Physics, Department of Mathematical Analysis, Sokolovsk\'a 49/83, 186 75 Praha 8, Czech Republic
}
\email{rancourt@karlin.mff.cuni.cz}

\author[T.~Kania]{Tomasz Kania}
\address[T.~Kania]{Mathematical Institute\\Czech Academy of Sciences\\\v Zitn\'a 25 \\115 67 Praha 1\\Czech Republic  and  Institute of Mathematics and Computer Science\\ Jagiellonian University\\ {\L}ojasiewicza 6, 30-348 Krak\'{o}w, Poland
}
\email{kania@math.cas.cz, tomasz.marcin.kania@gmail.com}

\author[J.~Swaczyna]{Jaros{\l}aw Swaczyna}
\address[J.~Swaczyna]{Mathematical Institute\\Czech Academy of Sciences\\\v Zitn\'a 25 \\115 67 Praha 1\\Czech Republic  and  Institute of Mathematics, {\L}\'od\'z University of Technology, W\'olcza\'nska 215, 93-005 {\L}\'od\'z, Poland}
\email{swaczyna@math.cas.cz, jaroslaw.swaczyna@p.lodz.pl}

\thanks{The two last-named authors acknowledge with thanks support received from GA\v{C}R project 19-07129Y; RVO 67985840.}
\date{\today}

\begin{document}
\maketitle
\begin{abstract}
    We prove that the coordinate functionals associated with filter bases in Banach spaces are continuous as long as the underlying filter is analytic. This removes the large-cardinal hypothesis from the result established by the two last-named authors ([Bull.~Lond. Math.~Soc.~53~(2021)]) at the expense of reducing the generality from projective to analytic. In particular, we obtain a ZFC solution to Kadets' problem of continuity of coordinate functionals associated with bases with respect to the filter of statistical convergence. Even though the automatic continuity of coordinate functionals beyond the projective class remains a mystery, we prove that a basis with respect to an arbitrary filter that has continuous coordinate functionals is also a basis with respect to a filter that is analytic.
\end{abstract}
\section{Introduction}
Filter bases generalise the familiar concept of a Schauder basis in a~Banach space in a~very natural way: let $\mathcal{F}$ be a filter of subsets of the natural numbers (a family of subsets that is closed with respect to finite intersections and is upwards-closed with respect to set inclusion; in order to avoid trivialities, we assume that every filter contains the filter of cofinite sets)
and let $X$ be a~Banach space (we consider Banach spaces over the fixed field $\mathbb K$ of real or complex numbers). A~sequence  $(e_n)_{n=1}^\infty$ in $X$ is an $\mathcal{F}$\emph{-basis} for $X$ whenever for each $x\in X$ there exists a unique sequence of scalars $(e^*_n(x))_{n=1}^\infty$ such that $x = \sum_{n,\mathcal{F}} e^*_n(x) e_n$, where the series is interpreted as the $\mathcal{F}$-limit of the sequence of the corresponding partial sums. (A sequence $(x_n)_{n=1}^\infty$ in a normed space $X$ $\FF$-\emph{converges} to $x\in X$ whenever for all $\varepsilon > 0$ there is $A\in \FF$ such that for all $n\in A$ one has $\|x - x_n\|<\varepsilon$.)\smallskip

Unlike for more familiar Schauder bases (\emph{i.e.}, filter bases with respect to the filter of cofinite sets), the question of continuity of the coordinate (linear) functionals $x\mapsto e^*_n(x)$ ($x\in X$) is non-obvious. This very problem was posed by Kadets and revived in 2011 with a particular emphasis on the filter of statistical convergence, that is the filter of subsets of $\IN$ having density $1$ (see, \emph{e.g.}, \cite{CGK,GanichevKadets}). Certainly, the problem is equivalent to the problem of continuity of the initial $\mathcal{F}$-basis projections $(P_n)_{n=1}^\infty$ given by $P_nx = \sum_{i=1}^n e_i^\star(x)e_i$ for $n\in \IN$ and $x\in X$. \smallskip


The present paper subsumes (at least in the analytic case) the results obtained in the recent study by the two last-named authors \cite{KS}, where it was proved that, assuming the existence of a supercompact cardinal, the coordinate functionals associated with a filter basis are continuous as long as the filter is projective. The aim of the paper is twofold. First, we remove the extra assumption of the existence of supercompact cardinals used for proving automatic continuity of the evaluation functionals of a filter basis at the expense of restricting to analytic ideals, which appear to be quite sufficient from the point of view of applications. In particular, it is sufficient to answer Kadets' question concerning the filter of statistical convergence in ZFC. Second, we contribute to (and hopefully trigger) further development of the theory of filter bases in Banach spaces, that may serve as a~proxy in the situation where genuine Schauder bases are unavailable. For instance, we prove the principle of small perturbations for such bases (Theorem~\ref{Th:C}) that is a basic result in the context of Schauder bases.\smallskip

We wish to emphasise that filter bases, despite striking similarities to Schauder bases, may behave quite differently.
\begin{itemize}
    \item Indeed, there is a filter basis (with respect to the filter of statistical convergence) in a separable Hilbert space whose initial projections are bounded, yet not uniformly (\cite[Example 1]{CGK}).
    \item Kochanek proved in \cite{Kochanek} that for filters generated by fewer sets than the pseudo-intersection number (in particular, for countably generated filters) one indeed retains the automatic continuity of initial basis projections. Moreover, for `filter-many' of them there is a uniform bound on the norm of these. However, the filter of statistical convergence is generated by precisely continuum many sets, so Kochanek's result does not apply to it.
    \item  Avil\'es, Kadets, P\'erez H\'ernandez, and Solecki \cite{AKHS} distilled  the key property from Ko\-cha\-nek's proof that makes it work and termed filters having it \emph{Baire filters}. However analytic Baire filters are necessarily countably generated (\cite[Proposition 4.1]{AKHS}).
\end{itemize}

\smallskip

A word of explanation behind topological properties of filters is required. Since there is a natural correspondence between subsets of the set of natural numbers and 0-1 sequences, one may treat filters on this set as subsets of the Cantor set $\{0,1\}^\IN$ equipped with the product topology and consider properties such as being Borel, analytic (continuous image
of a Polish space), etc.~that we shall freely do. For more information, we refer to \cite{Farah}, where the the dual notion of an ideal is considered. All unexplained terminology from Descriptive Set Theory (such as the definitions of the $\mathbf{\Sigma}^i_n$, $\mathbf{\Pi}^i_n$, and $\mathbf{\Delta}^i_n$-classes, $i=0,1$, $n\in\mathbb N$, and of Baire-measurability) is in line with \cite{Kechris}.\smallskip

Our first main result, Theorem~\ref{T: cont ZFC}, subsumes the main result of \cite{KS}.
\begin{mainth}\label{T: cont ZFC}
 Let $\FF$ be an analytic filter on $\IN$. Then for every $\FF$-basis the corresponding coordinate functionals are continuous.
\end{mainth}
As the filter of statistical convergence is Borel (\cite[Example 1.2.3(d)]{Farah}), hence analytic (by \cite[Theorem 13.7]{Kechris}), we obtain a~solution in ZFC to Kadets' problem.
\begin{corollary*}
 Let $\FF_{{\rm st}}$ be the filter of statistical convergence. Then for every $\FF_{{\rm st}}$-basis the corresponding coordinate functionals are continuous.
\end{corollary*}
Theorem~\ref{T: cont ZFC} has a natural generalisation in the case where all $\mathbf{\Delta}^1_n$-subsets of Polish spaces are Baire-measura\-ble. Since this assumption cannot be proved in ZFC alone, we decided to present the result as a separate theorem to highlight that Theorem~\ref{T: cont ZFC} is a theorem of ZFC, even though Theorem~\ref{T: cont ZFC} follows from Theorem A$^\prime$.
\begin{mainthprime*}\label{T: cont meas}
  Let $n\geqslant 1$. Assume that all $\mathbf{\Delta}^1_n$-subsets of Polish spaces are Baire-measurable. Suppose that $\FF$ is a $\mathbf\Sigma^1_n$-filter on $\IN$. Then for every $\FF$-basis the corresponding coordinate functionals are continuous.
 \end{mainthprime*}
A word of explanation on the assumption of Baire-measurability of $\mathbf{\Delta}^1_n$-sets is required. A classical result due to Banach and Mazur ensures that for $n \geqslant 1$  Baire-measurability of all $\mathbf{\mathbf{\Delta}}_{n+1}^1$-subsets (and even all $\mathbf{\Sigma}_{n+1}^1$-subsets) of Polish spaces follows from determinacy of $\mathbf{\Sigma}_{n}^1$-games on integers (this is an easy consequence of  \cite[Theorem 21.5]{Kechris}). The latter determinacy assumption has been shown by Martin \cite{MartinMeasurable} to follow for $n = 1$ from the existence of a measurable cardinal and by Martin and Steel \cite{MartinSteel} to follow for $n \geqslant 2$ from the existence of $n$ Woodin cardinals and a measurable cardinal above them. Knowing that supercompact cardinals are measurable (see \cite[Lemma 20.16]{Jech}) and limits of Woodin cardinals (see \cite[Exercise 34.1]{Jech}), it follows that the existence of a supercompact cardinal implies the determinacy of all projective games on integers and hence the Baire-measurability of projective subsets of Polish spaces. In particular, from Theorem A$^{\prime}$ we can recover the main result of \cite{KS}.
\begin{corollary*}[\cite{KS}]
Assume the existence of a supercompact cardinal. Let $\mathcal{F}$ be a projective filter on $\IN$. Then for every $\FF_{{\rm st}}$-basis the corresponding coordinate functionals are continuous.
\end{corollary*}
Nonetheless large cardinals are actually \emph{not} required to get Baire-measurability of projective sets. It was proved by Martin and Solovay \cite{MartinSolovay} that under Martin's Axiom and the negation of the Continuum Hypothesis all $\mathbf{\Sigma}_2^1$-subsets of Polish spaces are Baire-measurable. Moreover, it was proved by Shelah \cite{ShelahBP} that $\mathsf{ZFC}$ + `all projective subsets of Polish spaces are Baire-measurable' is equiconsistent with $\mathsf{ZFC}$ alone, thus completely removing the need for large cardinals. Results as those above are often proved in the literature for specific Polish spaces (\emph{e.g.}, the Cantor space), but here we freely use the fact that Baire-measurability results, if established for certain $\mathbf \Sigma^1_n$-classes ($n\in \mathbb N$) of subsets of the Cantor space, can be elementarily transferred to arbitrary Polish spaces using, \emph{e.g.}, \cite[Theorem 3.15]{CKW}, which asserts that any two uncountable Polish spaces without isolated points are Borel-isomorphic via a map that preserves meagre sets. \smallskip

We remark in passing that, assuming the consistency of one of the above-mentioned large-cardinal assumptions, the conclusion of Theorem A$^{\prime}$ is consistent with results obtained by forcings that do not destroy Woodin cardinals and measurable cardinals. It follows from \cite[Corollary]{HW} and \cite[Theorem 3]{LevySolovay} that forcings of reasonably small (accessible) cardinality satisfy this condition.\smallskip

 
Our second main result asserts that a filter basis with respect to a filter that is not analytic is also a filter basis with respect to a filter that actually is analytic, provided the coordinate functionals are continuous.
\begin{mainth}\label{T: reduction to analytic}
     Let $\FF$ be a filter on $\IN$ (not necessarily projective). Let $(e_n)_{n=1}^\infty$ be an $\FF$-basis with continuous coordinate functionals. Then there exists an analytic filter $\FF'\subset\FF$ such that $(e_n)_{n=1}^\infty$ is also an $\FF'$-basis.
\end{mainth}
Observe that, in the above result, the coordinate functionals of the basis $(e_n)_{n=1}^\infty$ are the same when considered with respect to the filter $\FF$ or the filter $\FF'$.

\section{Proofs of Theorems~\ref{T: cont ZFC} and A$^\prime$}

\begin{proof}[Proof of Theorem~\ref{T: cont ZFC}]
 Let $X$ be a Banach space with an $\mathcal{F}$-basis $(e_n)_{n=1}^\infty$. We denote by $(e_n^\star)_{n=1}^\infty$ the corresponding coordinate functionals. We start by proving that coordinate functionals $e_n^\star$ ($n\in \N$) are $\mathbf{\Sigma}_1^1$-measurable. Indeed, note that for a fixed open set $U\subset \IR$ we get 
\[
    \begin{aligned}
        e_n^\star(x) \in U &\iff \exists_{(\alpha_i)\in \IK^\IN} \left( \sum_{i, \FF} \alpha_i e_i = x \wedge \alpha_n \in U\right)\\
        & \iff  \exists_{(\alpha_i)\in \IK^\IN} \forall_{l \in \IN} \exists_{A \in \FF} \forall_{m \in A} \left(\Big\|\sum_{i=1}^m \alpha_i e_i - x\Big\| \leq \frac{1}{l} \wedge \alpha_n \in U \right)
 \end{aligned}
\]
By \cite[Proposition 3.3]{Kechris}, the space $X \times \IK^\IN \times \{0, 1\}^\IN$ is Polish. For $m \in \IN$ let us define $\Phi_m \colon X \times \IK^\IN \times \{0, 1\}^\IN \to \IR $ by 
\[\Phi_m(x,(\alpha_i),A)= \Big\|\sum_{i=1}^m \alpha_i e_i - x\Big\|\quad \big(x\in X, (\alpha_i)\in \mathbb K^{\mathbb N}, A\in \{0, 1\}^\IN\big),\] which is a continuous map. We then conclude that for every $l \in \IN$ the set $S_l$ given by
\[
\begin{aligned}
S_l =&  \left\{ (x, (\alpha_i), A) \in X \times \IK^\IN \times \{0, 1\}^\IN\colon \forall_{m \in A} \Big(\Phi_m\left(x,(\alpha_i),A\right) \leq \frac{1}{l} \wedge \alpha_n \in U \Big)\right\}
\\
 =&  \left\{ (x, (\alpha_i), A)\colon \forall_{m \in \IN} \Big[m \notin A \vee \big( \Phi_m\left(x,(\alpha_i),A\right) \leq \frac{1}{l}   \wedge \alpha_n \in U\big)\Big] \right\}
\\
 =&  \bigcap_{m \in \IN} \left\{ (x, (\alpha_i), A)\colon m \notin A \vee \Big( \Phi_m(x,(\alpha_i),A) \leq \frac{1}{l}   \wedge \alpha_n \in U \Big)\right\}
\\
= &\bigcap_{m \in \IN} \Bigg[\left\{ (x, (\alpha_i), A)\colon m \notin A \right\} \cup \Big( \big\{ (x, (\alpha_i), A)\colon \Phi_m\left(x,(\alpha_i),A\right) \leq \frac{1}{l}\big\} \\
& \cap \left\{ (x, (\alpha_i), A)\colon  \alpha_n \in U \right\} \Big) \Bigg]
\end{aligned}
\]
is a $G_\delta$ subset of $X \times \IK^\IN \times \{0, 1\}^\IN$ (the condition `$m \notin A$' encodes a set that is both closed and open in $\{0, 1\}^\IN$).\smallskip

Recall that countable intersections of analytic subsets of Polish spaces are analytic, and that images and preimages of analytic sets by a continuous maps between Polish spaces are analytic (see \cite[Proposition 14.4]{Kechris}). Also recall that Borel subsets of Polish spaces are analytic (see \cite[Theorem 13.7]{Kechris}). Since the filter $\FF$ is analytic, we deduce that its preimage $X \times \mathbb{K}^\IN \times \FF$ by the projection $X \times \mathbb{K}^\IN \times \{0, 1\}^\N \to \{0, 1\}^\N$ is also analytic, thus the set $$S:=\bigcap_{l \in \IN}\operatorname{proj}_{X \times \mathbb{K}^\IN}[X \times \mathbb{K}^\IN \times \FF \cap S_l]$$ is analytic. Observe that, for $x \in X$ and $(\alpha_i) \in \mathbb{K}^\IN$, we have
\[
    \begin{aligned}
       (x, (\alpha_i)) \in S &\iff \forall_{l \in \IN}\exists_{A \in \FF} (x, (\alpha_i), A) \in S \\
        & \iff  \forall_{l \in \IN} \exists_{A \in \FF} \forall_{m \in A} \left(\Big\|\sum_{i=1}^m \alpha_i e_i - x\Big\| \leq \frac{1}{l} \wedge \alpha_n \in U \right).
 \end{aligned}
\]
Thus we have $(e_n^*)^{-1}(U) = \operatorname{proj}_X[S]$, and we conclude that $(e_n^\star)^{-1}(U)$ is analytic.\smallskip

On the other hand, $e_n^\star$ is also $\mathbf\Pi_1^1$-measurable. Indeed, writing $U^c =: \bigcap_{n \in \N} V_n$, where the $V_n$'s are open subsets of $\mathbb{K}$, we obtain
\[
(e_n^\star)^{-1}(U)^c = \bigcap_{n \in \N}(e_n^\star)^{-1}(V_n),
\]
an analytic set. \smallskip

Once we know that for any open $U \subset \IK$ the set $(e_n^\star)^{-1}(U)$ is both analytic and coanalytic, by Souslin's Theorem (\cite[Theorem 14.11]{Kechris}), it is Borel, so $e_n^\star$ is Borel-measurable. Since each Borel set is Baire-measurable, we conclude that $e_n^\star$ is a Baire-measurable homomorphism between Polish groups $X$ and $\IK$. Consequently, by the Banach--Pettis Theorem (\cite[Theorem 9.10]{Kechris}) $e_n^\star$ is continuous.
\end{proof}
It was not essential for the proof of Theorem~\ref{T: cont ZFC} to prove $\mathbf{\Pi}_1^1$-measurablity of the coordinate functionals as all analytic sets are already Baire-measureable (see \cite[Theorem 21.6]{Kechris}). Nontheless, we decided to follow this way of reasoning in order to emphasise the analogy with the proof of Theorem A$^\prime$ whose proof is modelled on the proof of Theorem \ref{T: cont ZFC} and the latter follows from the former. However, for reader's convenience, we present a~sketch  of the proof with only certain details that are different.

\begin{proof}[Proof of Theorem A$^\prime$]
As in the proof of \ref{T: cont ZFC}, we fix an open $U \subset \IK$ and we get that
\[
    \begin{aligned}
        e_n^\star(x) \in U &\iff  \exists_{(\alpha_i)\in \IK^\IN} \forall_{l \in \IN} \exists_{A \in \FF} \forall_{m \in A} \left(\Big\|\sum_{i=1}^m \alpha_i e_i - x\Big\| \leq \frac{1}{l} \wedge \alpha_n \in U \right)
 \end{aligned}
\]
The family of $\mathbf{\Sigma}_n^1$-sets contains Borel sets, is closed under countable intersections, and under images and preimages with respect to continuous maps between Polish spaces (see \cite[Theorem 37.1]{Kechris}). So we can deduce similarly as in the proof of Theorem \ref{T: cont ZFC} that the set $(e_n^\star)^{-1}(U)$ is $\mathbf{\Sigma}_n^1$.\smallskip

Writing $U^c =: \bigcap_{n \in \N} V_n$, where the $V_n$'s are open sets, we have
\[
(e_n^\star)^{-1}(U)^c = \bigcap_{n \in \N}(e_n^\star)^{-1}(V_n),
\]
an $\mathbf{\Sigma}_n^1$-set. We deduce that $(e_n^\star)^{-1}(U)$ is $\mathbf{\Pi}_n^1$, so taking into account first part of proof we see that it is $\mathbf{\Delta}_n^1$. \smallskip

By the hypothesis, all $\mathbf{\Delta}_n^1$-sets are Baire-measurable, so $e_n^\star$ ($n\in\mathbb N$) is a Baire-measurable group homomorphism. Consequently, by the Banach--Pettis Theorem $e_n^\star$ ($n\in\mathbb N$) is continuous.
\end{proof}

 \section{Proof of Theorem~\ref{T: reduction to analytic}}
We are ready to prove Theorem~\ref{T: reduction to analytic}.

\begin{proof}[Proof of Theorem~\ref{T: reduction to analytic}]
Set 
\[
    \mc{A}:= \Bigg\{A \subset \IN\colon \exists_{x \in X} \exists_{\varepsilon>0}\; A \supset \Big\{ n \in \IN\colon \big\|\sum_{i=1}^n e_i^\star(x) e_i - x \big\| \leq \varepsilon \Big\}\Bigg\}.
\]
It follows from the fact that $(e_n)_{n=1}^\infty$ is an $\FF$-basis that $\mc{A} \subset \FF$. Note that $(e_n)_{n=1}^\infty$ remains an $\FF^\prime$-basis for any filter $\FF'$ on $\IN$ such that $\mc{A} \subset \FF' \subset \FF$. Now, for every $n \in \N$, consider the set
\[
\mc{B}_n= \Big\{ (x,\varepsilon,A)\in X \times (0,\infty)\times \{0, 1\}^\IN\colon n \in A \vee \big\|\sum_{i=1}^n e_i^\star(x) e_i - x \big\| > \varepsilon \Big\}.
\]
Using the continuity of $e_i^\star$ ($i\in \mathbb N$) and the fact that $\{0, 1\}^\IN$ comes equipped with the topology of pointwise convergence, we observe that $\mc{B}_n$ is open. Observing that
\[
\mc{A} = \on{proj}_{\{0, 1\}^\IN}\Big[\bigcap_{n \in \IN} \mc{B}_n\Big],
\]
we deduce that $\mc{A}$ is analytic.\smallskip

As $\mc{A} \subset \FF$, finite intersections on elements of $\mc{A}$ are readily non-empty. Consequently, $\mc{A}$ generates a filter $\FF' \subset \FF$. We have:
\[
    \begin{aligned}
    \FF^\prime= &\{A \supset \IN\colon \exists_{n \in \IN} \exists_{A_1, A_2, \ldots, A_n \in \mc{A}} A \supset A_1 \cap A_2 \cap \ldots \cap A_n \} \\
    = & \bigcup_{n \in \IN} \on{proj}_{(\{0, 1\}^\N)_1} [\{ (A, A_1, \ldots, A_n)\in (\{0, 1\}^\N)^{n+1}\colon  A \supset A_1 \cap \ldots \cap A_n \wedge A_1, \ldots, A_n \in \mc{A}\}],
\end{aligned}
\]
where $\on{proj}_{(\{0, 1\}^\N)_1}$ is the projection onto the first coordinate. Consequently, $\FF'$ is an analytic filter and $(e_n)_{n=1}^\infty$ is an $\FF'$-basis since $\mc{A}\subset \FF' \subset \FF$.
\end{proof}


\section{Closing remarks}
Having established continuity of basis projections (at least for analytic filters as proved in Theorem~\ref{T: cont ZFC}, which in the light of Theorem~\ref{T: reduction to analytic} appears to be sufficient for applications), it is possible to develop a filter basis theory analogous to the theory of Schauder bases. For instance, we may recover the  Bessaga--Pe{\l}czy\'nski small-perturbations principle.
\begin{mainth}\label{Th:C}
    Let $X$ be a Banach space with an $\mathcal{F}$-basis $(e_n)_{n=1}^\infty$, with continuous coordinate functionals (for example, $\mathcal{F}$ is analytic). If $(f_n)_{n=1}^\infty$ is a sequence in $X$ such that
        \[
            \sum_{n=1}^\infty {\|f_n-e_n\|\cdot \|e_n^*\|} \leqslant \delta < 1,
        \]
    for some $\delta\in (0,1)$, then $(f_n)_{n=1}^\infty$ is congruent to $(e_n)_{n=1}^\infty$, that is, $(f_n)_{n=1}^\infty$ is an $\mathcal{F}$-basis and the assignment $e_n \mapsto f_n$ $(n\in \mathbb N)$ extends to a surjective isomorphism of $X$ onto itself.

\end{mainth}
\begin{proof}
The proof are analogous to the familiar case of Schauder bases. Indeed, we observe that the operator $T\colon X\to X$ given by
$T(\sum_{i,\mathcal{F}}a_i e_i) =  \sum_{i,\mathcal{F}}(a_i(f_i - e_i))$ has norm strictly less than one. Indeed, for $x = \sum_{i,\FF} a_i e_i$ we have
\[
\begin{aligned}
\|Tx\|=& \Big\|\sum_{i,\FF} a_i (e_i-f_i)\Big\| 
\leq  \sum_{i}\| a_i (e_i-f_i)\|
= \sum_{i}\| e_i^\star (x) (e_i-f_i)\|\\
\leq & \sum_{i}\| e_i^\star \| \| x \| \| e_i-f_i \| = \Big( \sum_{i}\| e_i^\star \| \| e_i-f_i \| \Big) \| x \| \leqslant \delta \|x\|.
\end{aligned}
\]
Consequently, $I_X+T$ is invertible
and maps $e_i$ to $f_i$ ($i\in \mathbb N$), which proves that $(f_n)_{n=1}^\infty$ is an $\mathcal F$-basis too.
\end{proof}
Interestingly, general $\mathcal{F}$-bases are different from Schauder bases as, for example, we do not have the basis constant at the disposal. Indeed, it follows from Grunblum's criterion (\cite[Theorem V.1]{Diestel}) that if the basis projections are uniformly bounded, then the $\mathcal{F}$-basis is already a~Schauder basis.\smallskip

We conclude our note with two open problems.
%
\begin{question}
    Does there exists a Banach space $X$, a filter $\FF$ on $\IN$ and an $\FF$-basis $(e_n)_{n=1}^\infty$ for $X$ such that not all coordinate projections $(e_n^\star)_{n=1}^\infty$ are continuous?
\end{question}
Theorem \ref{T: reduction to analytic} demonstrates that filter bases with respect to analytic filters provide a very natural framework for the theory. In \cite[Example 1]{CGK} the authors provided an example of an $\FF_{{\rm st}}$-basis in $\ell_2$ that is not a Schauder basis. However, $\FF_{{\rm st}}$ has a relatively low complexity being $F_{\sigma\delta}$ (\cite[Example 1.2.3(d)]{Farah}). It is thus natural to ask whether being analytic in Theorem \ref{T: reduction to analytic} can be strengthened to being Borel.
\begin{question}
    Let $\FF$ be a filter and let $(e_n)_{n=1}^\infty$ be an $\FF$-basis for a Banach space $X$ with continuous coordinate functionals. Does there exists a Borel filter $\FF'$ such that $(e_n)_{n=1}^\infty$ is also an $\FF'$-basis? Should it be the case, what is the smallest complexity of such $\FF'$? 
\end{question}

\subsection*{Acknowledgements}
The authors are grateful to Ji\v{r}\'{\i} Spurn\'y and Miroslav Zelen\'y for organisning the 49\textsuperscript{th} Winter School in Abstract analysis in Sn\v{e}\v{z}n\'e, Czech Republic in January 2022 during which the research discussed in the present note was initiated. The second-named author thus cordially acknowledges support received from SONATA 15 No. 2019/35/D/\-ST1/\-01734 that allowed him to travel to this conference.

\bibliographystyle{plain}

\begin{thebibliography}{}
\bibitem{AKHS} A. Avil\'es, V. Kadets, A. P\'erez, and S. Solecki, Baire theorem for ideals of sets, \emph{J. Math. Anal. Appl.} \textbf{445} (2017), no. 2, 1221--1231.
\bibitem{CKW} J.~Cicho\'n, A.~Kharazishvili, and B.~W\c{e}glorz, \emph{Subsets of the Real Line - Part I}, Wydawnictwo Uniwersytetu \L \'{o}dzkiego, \L\'{o}dz, 1995. 
\bibitem{CGK} J.~Connor, M.~Ganichev, and V.~Kadets, A characterization of Banach spaces with separable duals via weak statistical convergence, \emph{J. Math. Anal. Appl.} \textbf{244} (2000), 251--261.
\bibitem{Diestel} J.~Diestel, \emph{Sequences and Series in Banach Spaces}, Graduate Texts in Mathematics, 92, Springer, New York, NY, 1984.
\bibitem{Farah} I.~Farah, Analytic quotients. Theory of lifting for quotients over analytic ideals on integers, {\it Mem. Amer. Math. Soc.}, \textbf{148} (2000), 177~pp.
\bibitem{GanichevKadets} M.~Ganichev and V.~Kadets, Filter convergence in Banach spaces and generalized bases, in: \emph{General Topology in Banach Spaces}, T. Banakh (ed.), Nova, 2001, 61--69.
\bibitem{HW} J.~Hamkins and H.~Woodin, Small forcing creates neither strong nor Woodin cardinals, \emph{Proc. Amer. Math. Soc.} \textbf{128} (2000), 3025--3029.   
\bibitem{Jech} T.~Jech, \emph{Set theory. The third millennium edition, revised and expanded}. Springer Monographs in Mathematics. Springer-Verlag, Berlin, 2003.
\bibitem{KS} T.~Kania and J.~Swaczyna, Large cardinals and continuity of coordinate functionals of filter bases in Banach spaces. \emph{Bull. Lond. Math. Soc}. \textbf{53} (2021), no. 1, 231--239.
\bibitem{Kechris} A.~S.~Kechris, \emph{Classical descriptive set theory}, vol. 156 of Graduate Texts in Mathematics, Springer-Verlag, New York, 1995.
\bibitem{Kochanek} T.~Kochanek, $\mathcal F$-bases with brackets and with individual brackets in Banach spaces, \emph{Studia Math.} \textbf{211} (2012), 259--268.
\bibitem{LevySolovay}
A.~L\'evy and R.~M.~Solovay, Measurable cardinals and the continuum hypothesis, \emph{Israel J. Math} \textbf{5}, No. 4 (1967), 234--248.
\bibitem{MartinMeasurable}
D.~A.~Martin, Measurable cardinals and analytic games, \emph{Fund. Math.} \textbf{66}, No. 3 (1970), 287--291.
\bibitem{MartinSolovay}
D.~A.~Martin and R.~M.~Solovay, Internal {C}ohen extensions, \emph{Ann. Math. Logic.} \textbf{2}, No. 2 (1970), 134--178.
\bibitem{MartinSteel} D.~A.~Martin and J.~R.~Steel, A proof of projective determinacy, \emph{J. Amer. Math. Soc.} \textbf{2}, No. 1 (1989), 71--125.
\bibitem{ShelahBP} S.~Shelah, Can you take {S}olovay's inaccessible away?, \emph{Isr J. Math} \textbf{48}, No. 1 (1984), 1--47.

\end{thebibliography}

\end{document}